\documentclass[oneside]{amsart}

\usepackage{amsmath}
\usepackage{amsfonts}
\usepackage{mathrsfs}
\usepackage{amsthm}
\usepackage{amsbsy}
\usepackage{amssymb}
\usepackage[margin=1in]{geometry}
\usepackage{graphicx}
\usepackage{cleveref}  

\theoremstyle{plain}
\newtheorem{theorem}{Theorem}[section]
\newtheorem{proposition}[theorem]{Proposition}
\newtheorem{lemma}[theorem]{Lemma}
\newtheorem{corollary}[theorem]{Corollary}
\newtheorem{remark}[theorem]{Remark}
\newtheorem{definition}[theorem]{Definition}
\newtheorem{notation}[theorem]{Notation}
\newtheorem{example}[theorem]{Example}
\newtheorem{conjecture}[theorem]{Conjecture}

\newcommand{\PP}{\mathbb{P}}
\DeclareMathOperator{\ord}{ord}
\newcommand{\pf}{ \mathfrak{p}}
\DeclareMathOperator{\Tail}{Tail}
\DeclareMathOperator{\Per}{Per}
\DeclareMathOperator{\PrePer}{PrePer}

\begin{document}


\author{Sebastian Troncoso}
\address{Michigan State University, Michigan, USA}
\curraddr {619 Red Cedar Road
East Lansing, MI 48824, USA}
\email{troncosomath@gmail.com}

\title{Bounds for preperiodic points for maps with good reduction}
\date{}  

\begin{abstract}
Let $K$ be a number field and let $\phi$ in $K(z)$ be a rational function of degree $d\geq 2$. Let $S$ be the places of bad reduction for $\phi$ (including the archimedean places). Let $\Per(\phi,K)$, $\PrePer(\phi, K)$, and $\Tail(\phi,K)$ be the set of $K$-rational periodic, preperiodic, and purely preperiodic points of $\phi$, respectively.
The present paper presents two main results. The first result is a bound for $|\PrePer(\phi,K)|$ in terms of the number of places of bad reduction $|S|$ and the degree $d$ of the rational function $\phi$. This bound significantly improves a previous bound given by J. Canci and L. Paladino.
For the second result, assuming that  $|\Per(\phi,K)| \geq 4$ (resp.\ $|\Tail(\phi,K)| \geq 3$), we prove bounds for $|\Tail(\phi,K)|$ (resp.\ $|\Per(\phi,K)|$) that depend only on the number of places of bad reduction $|S|$ (and not on the degree $d$). We show that the hypotheses of this result are sharp, giving counterexamples to any possible result of this form when $|\Per(\phi,K)| < 4$ (resp.\ $|\Tail(\phi,K)| < 3$).

\end{abstract}



\maketitle


\section{Introduction}

Let $K$ be a number field and let $\phi \in K(z)$ be a rational function. Let $\phi^n$ denote the $n^{th}$ iterate of $\phi$ under composition and $\phi^0$ the identity map. The \emph{orbit} of  $P\in \PP^1(K)$ under $\phi$ is the set  $O_\phi(P)=\{\phi^n(P )  :  n \geq 0 \}$. A point $P \in \PP^1(K)$ is called \emph{periodic} under $\phi$ if there is an integer $n > 0$ such that $\phi^n(P)=P$.  It is called \emph{preperiodic} under $\phi$ if there is an integer $m \geq 0$ such that $\phi^m(P)$ is periodic. A point that is preperiodic but not periodic is called a \emph{tail} point. Let $\Tail(\phi,K)$, $\Per(\phi,K)$ and $\PrePer(\phi,K)$ be the sets of $K$-rational tail, periodic and preperiodic points of $\phi$, respectively.

For any morphism $\phi : \PP^N \rightarrow \PP^N$ of degree $d\geq 2$, Northcott \cite{Northcott1950} proved in 1950 that the total number of $K$-rational preperiodic points of $\phi$ is finite. In fact, from Northcott's proof, an explicit bound can be found in terms of the
coefficients of $\phi$. In 1994, Morton and Silverman \cite{MPS1994} conjectured that $|\PrePer(\phi,K)|$ can be bounded in terms of only a few basic parameters.

\begin{conjecture} [Uniform Boundedness Conjecture] $\quad$ \\
Let $K$  be a number field with $[K:\mathbb{Q}]=D$, and let $\phi $ be an endomorphism of $\PP^N$, defined over $K$. Let $d \geq 2$ be the degree of $\phi$. Then there is $C=C(D,N,d)$ such that $\phi$ has at most $C$ preperiodic points in $\PP^N(K)$.
\end{conjecture}

The conjecture seems extremely difficult to prove even in the simpler case when $(K,N,d)=(\mathbb{Q},1,2)$. Further, in this case, explicit conjectures have been formulated. For instance, Poonen \cite{Poonen1998} conjectured an explicit bound when $\phi$ is a quadratic polynomial map over $\mathbb{Q}$. Since every such quadratic polynomial map is conjugate to a polynomial of the form $\phi_c(z)=z^2+c$ with $c\in \mathbb{Q}$ we can state Poonen's conjecture as follows: Let $\phi_c \in \mathbb{Q}[z]$ be a polynomial of degree 2 of the form $\phi_c(z)=z^2+c$ with $c\in \mathbb{Q}$. Then $ | \PrePer(\phi_c,\mathbb{Q}) | \leq 9$. B. Hutz and P. Ingram \cite{HI2013} have shown that Poonen's conjecture holds when the numerator and denominator of $c$  don't exceed  $10^8$.

This work has two main contributions. The first result gives a bound for $|\PrePer(\phi,K)|$ in terms of the number of places of bad reduction $|S|$ and the degree $d$ of the rational function $\phi$. This bound significantly improves a previous bound given by J. Canci and L. Paladino  \cite{CP2014}.

In the second result, assuming that  $|\Per(\phi,K)| \geq 4$ (resp. $|\Tail(\phi,K)| \geq 3$), we prove bounds for $|\Tail(\phi,K)|$ (resp.\ $|\Per(\phi,K)|$) that depend only on the number of places of bad reduction $|S|$  and $[K:\mathbb{Q}]$ (and not on the degree $d$). We show that the hypotheses of this result are sharp, \Cref{example1} and \Cref{example2} give counterexamples to any possible result of this form when $|\Per(\phi,K)| < 4$ (resp.\ $|\Tail(\phi,K)| < 3$).

\begin{theorem}
Let $K$ be a number field and $S$ a finite set of places of $K$ containing all the archimedean ones. Let $\phi $ be an endomorphism of $\PP^1$, defined over $K$, and $d \geq 2$ the degree of $\phi$. Assume $\phi$ has  good reduction outside $S$.
\begin{enumerate}

\item [(a)] \label{th 3 periodic}
If there are at least three $K$-rational tail points of $\phi$ then
$$|\Per(\phi,K)| \leq 2^{16|S|}+3. $$

\item [(b)] \label{th 4 preperiodic}
If there are at least four $K$-rational periodic points of $\phi$ then
$$|\Tail(\phi,K)| \leq 4(2^{16|S|}).$$
\end{enumerate}
\end{theorem}

Using the previous theorem, we can deduce a bound for $\PrePer(\phi,K)$ in terms of $|S|$ and the degree of $\phi$ for any endomorphism of $\PP^1$.

\begin{corollary} \label{theorem 3}
Let $K$ be a number field and $S$ a finite set of places of $K$ containing all the archimedean ones. Let $\phi $ be an endomorphism of $\PP^1$, defined over $K$, and $d \geq 2$ the degree of $\phi$. Assume $\phi$ has  good reduction outside $S$. Then
\begin{enumerate}
\item [(a)] $|\Per(\phi,K)| \leq  2^{16|S|d^3}+3.$

\item [(b)] $|\Tail(\phi,K)| \leq  4(2^{16|S|d^3}) .$

\item [(c)] $|\PrePer(\phi,K)| \leq 5(2^{16|S|d^3})+3.$

\end{enumerate}

\end{corollary}

These bounds depend, ultimately, on a reduction to $S$-unit equations. Using a reduction to Thue-Mahler equations instead, we obtain a better bound for $|\Tail(\phi,K)|$ in terms of $|S|$ and $d$.

\begin{theorem} \label{theorem 4}
Let $K$ be a number field and $S$ a finite set of places of $K$ containing all the archimedean ones. Let $\phi $ be an endomorphism of $\PP^1$, defined over $K$, and $d \geq 2$ the degree of $\phi$. Assume $\phi$ has  good reduction outside $S$. Then
$$|\Tail(\phi,K) | \leq d\max\left\{  (5 * 10^6 (d^3+1))^{|S|+4} ,4(2^{64(|S|+3)}) \right \}.$$

\end{theorem}

To get a similar  bound for $|\Per(\phi,K)|$ we need to assume that $\phi$ has at least one $K$-rational tail point. Under this assumption, using Theorem \ref{th 4 preperiodic} and results about Thue-Mahler equation, we can get:

\begin{theorem} \label{theorem 5}
Let $K$ be a number field and $S$ a finite set of places of $K$ containing all the archimedean ones. Let $\phi $ be an endomorphism of $\PP^1$, defined over $K$, and $d \geq 2$ the degree of $\phi$. Assume $\phi$ has  good reduction outside $S$. If $\phi$ has at least one $K$-rational tail point then
$$|\Per(\phi,K)| \leq   \max \left\{  (5 * 10^6 (d-1))^{|S|+3} ,4(2^{128(|S|+2)}) \right\}+1.$$
\end{theorem}

While the work described in this paper was being carried out, Canci and Vishkautsan \cite{CS2016} proved a bound for $|\Per(\phi,K)|$, just assuming that $\phi$ has  good reduction outside $S$. Their bound on $|\Per(\phi,K)|$ is roughly of the order of $d2^{16|S|}+2^{2187|S|}$ where  $d \geq 2$ is the degree of $\phi$.


Let's recall previous bounds for $|\PrePer(\phi,K)|$ which are relevant for our work. In 2007, Canci \cite{Canci2007} proved for rational functions with good reduction outside $S$  that the length of finite orbits is bounded by:
\begin{equation} \label{eq1}
\left [e^{10^{12}}(|S|+1)^8(\log(5(|S|+1)))^8 \right ]^{|S|}.
\end{equation}
Note that this bound depends only on the cardinality of $S$.


In Canci's recent work (2014) with Paladino \cite{CP2014} a sharper bound for the length of finite orbits was found:
\begin{equation} \label{eq2}
\max\left \{ (2^{16|S|-8}+3)  \left [12|S|\log(5|S|)\right ]^{[K:\mathbb{Q}]},  [12(|S|+2)\log(5|S|+5)]^{4[K:\mathbb{Q}  ] }  \right \} .
\end{equation}


In our work we are interested in the number of $K$-rational tail points and $K$-rational periodic points,  $|\Tail(\phi,K)|$ and $|\Per(\phi,K)|$ respectively.

The bounds mentioned in (\ref{eq1}) and (\ref{eq2}) can be used to deduce bounds on $|\PrePer(\phi,K)|$. For instance, if we assume that every finite orbit has cardinality given by (\ref{eq1}) and using that every point could have at most $d$ preimages under $\phi$ we obtain a bound for $|\PrePer(\phi,K)|$ that is roughly of the order of $d^{|S|^{8|S|}\log|S|}$ where  $d \geq 2$ is the degree of $\phi$. Similarly, the bound deduced from (\ref{eq2}) is roughly of the order of $d^{2^{16|S|}(|S|\log(|S|)^{[K:\mathbb{Q}]}}$, where  $d \geq 2$ is  the degree of $\phi$. These bounds are  polynomial in the degree of $\phi$, however they will be rather large in terms of $|S|$.

In 2007, Benedetto \cite{Benedetto2007} proved for the case of polynomial maps of degree $d \geq 2$ that $|\PrePer(\phi,K)|$ is bounded by $O(|S| \log|S|)$ , where $S$ is the set of places of $K$ at which $\phi$ has bad reduction, including all archimedean places of $K$. The big-$O$ is essentially $\frac{d^2-2d+2}{\log d}$ for large $|S|$. Many other results have been proven in recent years \cite{BCHKW2014}, �\cite{Canci2010} , \cite{MN2006}, \cite{Poonen1998}.

We end this introduction with a brief outline of the rest of the paper.  Section~\ref{prelim} introduces some classical notation and definitions from arithmetic dynamics along with some propositions needed for the main theorems of the paper. In particular, Corollary \ref{cor. S-unit} will play a crucial role in almost every proof. The corollary states that the $\pf$-adic logarithmic distance between a $K$-rational tail point and a $K$-rational periodic point is 0 up to a few exceptions.

Section~\ref{main} presents the proof for Theorem \ref{th 3 periodic} and Corollary \ref{theorem 3} using Corollary \ref{cor. S-unit} together with the $S$-unit theorem. 

Section~\ref{main2} presents an improvement in the bound found in Corollary \ref{theorem 3} for $|\Tail(\phi,K)|$. This new bound is polynomial in the degree of $\phi$ and exponential in  the cardinality of $S$. The main idea for obtaining this new bound is to substitute the arguments involving $S$-unit equations with arguments involving Thue-Mahler equations. This appears to give the best known general bound for the cardinality of $|\Tail(\phi,K)|$. In this section we also provide a new bound for $|\Per(\phi,K)|$ with a small hypothesis on $\phi$.

Finally, Section~\ref{Eg} presents two examples related to our results. Specifically, we give an example to show that a bound of $|\Tail(\phi,K)|$ must depend on the degree of $\phi$ when $\phi$ has three or fewer $K$-rational periodic points. Similarly, the bound of $|\Per(\phi,K)|$ must depend on the degree of $\phi$ when $\phi$ has two or fewer $K$-rational tail points.

\section*{Acknowledgement} 
The author would like to thank his adviser Dr. Aaron Levin for all his help. The author would also like to thank Jung Kyu Canci, Casey Machen, Charlotte Ure and Solomon Vishkautsan for reviewing previous versions.

\section{Preliminaries} \label{prelim}

\subsection{Notation and definitions}

\begin{notation}
In the present article we will use the following notation:

$K$ a number field;

$\bar{K}$ an algebraic closure of $K$;

$R$ the ring of integers of $K$;

$\pf$ a non-zero prime ideal of $R$;

$v_{\pf}$ the $\pf$-adic valuation on $K$ corresponding to the prime ideal $\pf$ (we always assume

$v_{\pf}$ to be normalized so that  $v_{\pf}(K^*)=\mathbb{Z}$);

If the context is clear, we will also use $v_{\pf}(I)$ for the $\pf$-adic valuation of a fractional ideal $I$ of $K$;

$S$ a fixed finite set of places of $K$ including all archimedean
places;

$|S|=s$ the cardinality of $S$;

$R_S = \{x \in K : v _\pf (x) \geq 0 \mbox{    for every prime ideal     } \pf \notin S\}$ the ring of $S$-integers;

$R^{*}_S = \{x \in K : v _\pf (x) = 0 \mbox{    for every prime ideal     } \pf \notin S\}$ the group of $S$-units;


$\Per(\phi,K)$ the set of $K$-rational periodic points;

$\Tail(\phi,K)$ the set of $K$-rational tail points;

$\PrePer(\phi,K)$ the set of $K$-rational preperiodic points.
\end{notation}

We begin by recalling the definition of the $\pf$-adic logarithmic distance between two points in $\PP^1$.

\begin{definition}
Let $P_1=[x_1:y_1]$ and $P_2=[x_2:y_2]$ be points in $\mathbb{P}^1(K)$. We will denote by

$$\delta_\pf(P_1,P_2)=v_{\pf}(x_1y_2-x_2y_1)-\min\{v_{\pf}(x_1),v_{\pf}(y_1) \} -\min\{v_{\pf}(x_2),v_{\pf}(y_2) \} $$
the $\pf$-adic logarithmic distance between the points $P_1$ and $P_2$.
\end{definition}

Note that $\delta_\pf (P_1,P_2)$ is independent of the choice of homogeneous coordinates. We use the convention that $v_{\pf}(0)=\infty$. Properties of the $\pf$-adic logarithmic distance can be found in  \cite{MPS1995} and \cite{Silverman2007}. The following definition introduces the idea of normalized forms with respect to $\pf$.

\begin{definition}
\begin{enumerate}

\item We say that $P =[x:y] \in\PP^1(K)$ is in normalized form with respect to $\pf$ if
$$\min\{ v_\pf(x), v_\pf(y)\} = 0 .$$

\item Let $\phi$ be an endomorphism of $\PP^1$, defined over $K$. Assume $\phi$ is given by
$$ \phi = [F(X,Y):G(X,Y)]$$
where $F,G \in K[X,Y]$ are homogeneous polynomials with no common factors. We say that the pair $(F,G)$ is normalized with respect to $\pf$ or that $\phi$ is in normalized form with respect to $\pf$ if $F,G\in R_\pf[X,Y]$ and at least one coefficient of $F$ or $G$ is not in the maximal ideal of $R_\pf$. Equivalently, $\phi=[F:G]$ is normalized with respect to $\pf$ if
$$F(X,Y)=a_0X^d+a_1X^{d-1}Y+...+a_{d-1}XY^{d-1}+a_dY^d $$
and
$$G(X,Y)=b_0X^d+b_1X^{d-1}Y+...+b_{d-1}XY^{d-1}+b_dY^d $$
satisfy
$$\min\{v_\pf(a_0),..., v_\pf(a_d), v_\pf(b_0), ....., v_\pf(b_d)\} =0.$$

\end{enumerate}
\end{definition}

\begin{remark} \label{Remark2.5}
Note that if $P=[x_1:x_2]$ and $Q=[y_1:y_2]$ are in normalized form with respect to $\pf$ then $\delta_\pf(P_1,P_2)=v_{\pf}(x_1y_2-x_2y_1).$
\end{remark}

Since $R_\pf$ is a discrete valuation ring, we can always find a representation of $P$ and $\phi$ in normalized form with respect to $\pf$. However, it is not always true that the same representation is normalized for every $\pf$. For this reason we need a more global definition of normalized forms.

\begin{definition}
\begin{enumerate}

\item We say that $P =[x:y] \in\PP^1(K)$ is normalized with respect to $S$ if $[x:y]$ is normalized with respect to $\pf$ for every $\pf \notin S$.

\item Let $\phi=[F:G]$ be an endomorphism of $\PP^1$, defined over $K$. We say that $\phi$ is normalized with respect to $S$ if $[F:G]$ is normalized with respect to $\pf$ for every $\pf \notin S$.
\end{enumerate}
\end{definition}

\begin{remark} \label{normalize=principal}
Notice that a point $P =[x:y] \in\PP^1(K)$ admits a normalized form with respect to $S$ if and only if the $R_S$-fractional ideal $(x,y)$ is principal.

\end{remark}

Since the concept of good reduction is present through the entire paper, we will recall the definition.

\begin{definition}
Let $\phi $ be an endomorphism of $\PP^1$, defined over $K$ and write $\phi=[F:G]$ in normalized form with respect to $\pf$. We say that $\phi$ has good reduction at $\pf$ if $\tilde{F}(X,Y)=\tilde{G}(X,Y)=0$ has no solutions in $\PP^1(\bar{k})$, where $\tilde{F}$ and $\tilde{G}$ are the reductions of $F$ and $G$ modulo $\pf$ respectively and $k$ is the residue field of $R_\pf$.

We say that $\phi$ has good reduction outside $S$ if $\phi$ has good reduction at $\pf$ for every $\pf \notin S$.
\end{definition}

We also recall two facts on the relation between good reduction and normalized form.

\begin{remark} \label{Silverman Theorem 2.18a}[\cite{Silverman2007}, p.59.]
Let $\phi $ be an endomorphism of $\PP^1$, defined over $K$ and write $\phi=[F:G]$ in normalized form with respect to $\pf$.  If $\phi$ has good reduction at $\pf$ then $\phi^n$ has good reduction at $\pf$ for every $n \geq 2$. Even more, $\phi^n=[F_n: G_n]$ is in normalized form with respect to $\pf$, where $F_n(X,Y)=F(F_{n-1}(X,Y),G_{n-1}(X,Y))$ , $G_n(X,Y)=G(F_{n-1}(X,Y),G_{n-1}(X,Y))$ , $F_1(X,Y)=F(X,Y)$,  and $G_1(X,Y)=G(X,Y)$.
\end{remark}

\begin{remark} \label{Silverman Theorem 2.18b} [\cite{Silverman2007}, p.59.]
Let $\phi $ be an endomorphism of $\PP^1$, defined over $K$ and write $\phi=[F:G]$ in normalized form with respect to $\pf$.  Let $P=[a:b] \in  \PP^1(K)$ be in normalized form with respect to $\pf$. If $\phi$ has good reduction at $\pf$, then $[F(a,b): G(a,b)]$ is in normalized form with respect to $\pf$.
\end{remark}

Finally we give a more explicit definition of a $K$-rational tail point.

\begin{definition}
Given a periodic point $P \in \PP^1(K)$, we say that a point $Q\in \PP^1(K)$ is in the tail of $P$ if it is preperiodic but not a periodic point  and $P$ is in the orbit of $Q$.

We say that $Q$ is a tail point if it is in the tail of some periodic point.
 \end{definition}

\subsection{Results from diophantine geometry and arithmetic dynamics}
\hspace{1mm}
\vspace{3mm}

Bounding the number of solutions of important equations has always been a fascinating problem. In particular, one can study this problem when the solutions come from the group of $S$-units of a number field $K$.

We can consider the $S$-unit equation $ax+by=1$ where $a,b \in K^{*}$ and $x,y$ are $S$-units. Bounds on the number of solutions of this equation give powerful consequences in different areas of mathematics.  Among many studies on the $S$-unit equation, one of the best bounds is the following:

\begin{theorem} [Beukers and Schlickewei  \cite{BS1996}]
Let $\Gamma$ be a subgroup of $(K^{*})^2=K^{*}\times K^{*}$ of rank $r$. Then the equation
$$x+y=1  \quad\quad \mbox{in  }\quad (x,y) \in \Gamma$$
has at most $2^{8(r+1)}$ solutions.
\end{theorem}

\begin{corollary} \label{S-unit}
Let $\Gamma_0$ be a subgroup of $K^{*}$ of rank $r$. Consider $\Gamma=\Gamma_0 \times \Gamma_0$ and assume $a,b \in K^{*}$. Then the equation
$$ax+by=1  \quad\quad \mbox{in  }\quad (x,y) \in \Gamma$$
has at most $2^{8(2r+2)}$ solutions.
\end{corollary}

We will recall similar results on the closely related Thue-Mahler equation.

Let $F(X,Y)$ be a binary  form of degree $r \geq 3$ with coefficients in $R_S$. An $R_S^{*}$-coset of solutions of
\begin{equation} \label{T-M}
 F(x,y) \in R_S^{*}  \quad\quad \mbox{ in } \quad (x,y) \in R_S^{2}
\end{equation}
is a set $\{\epsilon(x,y): \epsilon \in R_S^{*} \}$, where $(x,y)$ is a fixed solution of (\ref{T-M}).

\begin{theorem}[Evertse \cite{Evertse1997}]    \label{Thue-Mahler}
Let $F(X,Y)$ be a binary  form of degree $r \geq 3$ with coefficients in $R_S$ which is irreducible over $K$. Then the set of solutions of
$$ F(x,y) \in R_S^{*}  \quad\quad \mbox{ in } \quad (x,y) \in R_S^{2} $$
is the union of at most
$$ (5*10^6r)^s$$
$R_S^{*}$-cosets of solutions.
\end{theorem}

Next we will give the definition and some results on the $n^{th}$ dynatomic polynomial associated to a rational function $\phi$.

\begin{definition}
Let $\phi(z)\in K(z)$ be a rational function of degree $d$. For any $n \geq 0$ write
$$\phi^n(X,Y)=[F_n(X,Y):G_n(X,Y)]  $$
with homogeneous polynomials $F_n,G_n \in K[X,Y]$ of degree $d^n$. The $n$-period polynomial of $\phi$ is the polynomial
$$\Phi_{\phi,n}(X,Y)=YF_n(X,Y)-XG_n(X,Y).  $$
$\Phi_{\phi,n}$ is well defined up to a constant.
Notice that $\Phi_{\phi,n}(P)=0$ if and only if $\phi^n(P)=P$.

The $n^{th}$ dynatomic polynomial of $\phi$ is the polynomial
\begin{equation*}
\Phi^{*}_{\phi,n}(X,Y)=\displaystyle\prod_{k|n}(YF_k(X,Y)-XG_k(X,Y))^{\mu(n/k)}= \displaystyle\prod_{k|n} \Phi_{\phi,k}(X,Y)^{\mu(n/k)}
\end{equation*}

where $\mu$ is the M{\"o}bius function. If $\phi$ is fixed, we write $\Phi_n$ and $\Phi_n^{*}$ for $\Phi_{\phi,n}$ and $\Phi^{*}_{\phi,n}$.
\end{definition}

The following remark will give us the degree of the dynatomic polynomial which will be useful in the end of the next section.

\begin{remark} \label{degree dynatomic}
The degree of the $n^{th}$ dynatomic polynomial is given by 

\begin{equation*}
\deg (\Phi^{*}_{\phi,n}) = \displaystyle\sum_{k|n} \mu\left(\frac{n}{k}\right) (d^k+1)
\end{equation*}

In particular, if $n=1$ the degree of $\Phi^{*}_{\phi,n}$ is $d+1$ and if $n$ is a prime number then the degree of $\Phi^{*}_{\phi,n}$ is $d^n-d$.

\end{remark}

\begin{definition}
Let $\phi(z) \in K(z)$ be a rational function and $P \in \PP^1(K)$. We say that $P$ has formal period $n$ if $\Phi^{*}_{\phi,n}(P)=0$.
\end{definition}

\begin{definition}
Let $\phi(z) \in K(z)$ be a rational function of degree $d \geq 2$  and $P \in \PP^1(K)$. We say that $P$ has primitive period $n$ if $\phi^n(P)=P$ and $\phi^i(P) \neq P$ for all $1 \leq i < n$ .
\end{definition}

\begin{theorem}   [\cite{Silverman2007} , p.151]   \label{dynatomic}
Let $\phi(z) \in K(z)$ be a rational function of degree $d \geq 2$. For each $P \in \PP^1(\bar{K})$, let
$$a_P(n)=\ord_P(\Phi_{\phi,n}(X,Y)) \quad \mbox{and} \quad a_P^{*}(n)=\ord_P(\Phi^{*}_{\phi,n}(X,Y))$$
where $\ord_P(\Phi_{\phi,n}(X,Y))$ and $\ord_P(\Phi^{*}_{\phi,n}(X,Y))$ are the order of zero or pole at $P$ of $\Phi_{\phi,n}(X,Y)$ and $\Phi^{*}_{\phi,n}(X,Y)$, respectively. Then
\begin{enumerate}
\item[(a)] $\Phi^{*}_{\phi,n}\in K[X,Y]$, or equivalently,
$$a_P^{*}(n) \geq 0 \mbox{    for all   } n \geq 1 \mbox{   and all   } P\in \PP^1. $$

\item[(b)] Let $P$ be a point of primitive period $m$ and let $\lambda (P)=(\phi^m)'(P)$ be the multiplier of $P$. Then  $P$ has formal period $n$, i.e., $a_P^{*}(n) > 0$, if and only if one of the following is true:
\begin{enumerate}
\item[(i)] $n=m$

\item[(ii)] $n=mr$ and $\lambda(P)$ is a primitive $r^{th}$ root of unity.

\end{enumerate}
\end{enumerate}
In particular,  $a_P^{*}(n) $ is nonzero for at most two values of $n$.
\end{theorem}

We will recall a result on the existence of $n$-periodic points for rational functions due to Baker.

\begin{theorem} [Baker \cite{Baker1964}] \label{Baker}
Let $\phi(z) \in K(z)$ be a rational function of degree $d \geq 2$ defined over $K$. Suppose that $\phi$ has no primitive $n$-periodic points. Then $(n,d)$ is one of the pairs
$$(2,2),(2,3),(3,2),(4,2). $$
If $\phi$ is a polynomial, then only $(2,2)$ is possible.
\end{theorem}

\begin{remark} \label{Baker remark}
Kisaka completely classifies all the rational functions associated to the exceptional pairs $(n,d)$ mentioned in Baker's Theorem. Each of these exceptional rational functions has at least two distinct fixed points in $\bar{K}$.
\end{remark}

To end this subsection we will state a strong consequence of Dirichlet's Theorem on  primes in arithmetic progression.

\begin{theorem}[\cite{Paulo2001}, p.527]  \label{ideal class}
If $I$ is a fractional ideal of $R_{S}$, then there is a prime ideal $P_0$ of $R_{S}$ such that $[I] = [P_0]$ as $R_{S}$-ideal classes i.e. there is a $\lambda \in K$ such that $I=(\lambda)P_0 $.
\end{theorem}

\subsection{Main propositions}
\hspace{1mm}
\vspace{3mm}

The next proposition is a fundamental ingredient for the entire paper.
\begin{proposition}
Let $\phi$ be an endomorphism of $\PP^1$, defined over $K$. Suppose $\phi$ has good reduction outside $S$. Let $P\in\PP^1(K)$ be a periodic point, $Q\in\PP^1(K)$ a fixed point with $P\neq Q$ and $R\in\PP^1(K)$ a tail point of $Q$. Then $\delta_\pf(P,R)=0$ for every $\pf \notin S$.
\end{proposition}

\begin{proof}
Let $\pf \notin S$ be a prime of good reduction. Consider $P=[p_1:p_2],Q=[q_1:q_2],R=[r_1:r_2]$ and $\phi=[F(x,y):G(x,y)]$ all in normalized form with respect to $\pf$. Let $n$ be the period of $P$ and $L_Q(x,y)=q_2x-q_1y$ a linear form defining $Q$ .

Given $N>1$ consider  $\phi^{N}=[F_{N}(x,y):G_{N}(x,y)]$ where $F_{N}(x,y)=F_{N-1}(F(x,y),G(x,y)) $ , $G_{N}(x,y)=G_{N-1}(F(x,y),G(x,y))$ , $F_1(x,y)=F(x,y)$ and $G_1(x,y)=G(x,y)$. By Remark \ref{Silverman Theorem 2.18a}, 
$(F_{N} , G_{N})$ is in normalized form with respect to $\pf$ and by Remark \ref{Silverman Theorem 2.18b}, $[F_N(p_1,p_2) : G_N(p_1,p_2)]$ is in normalized form with respect to $\pf$ \emph{i.e.} $\min \{ v_\pf(F_{N}(p_1,p_2)) , v_\pf(G_{N}(p_1,p_2))  \}=0$.

Therefore for every $m>0$ we can find $\lambda \in (R)^{*}_\pf$ such that $F_{nm}(p_1,p_2)= \lambda p_1 $ and $G_{nm}(P)=  \lambda p_2 $.  We conclude
\begin{equation} \label{Eq2.15}
v_\pf(L_Q  (F_{nm}(p_1,p_2),G_{nm}(p_1,p_2)))=v_\pf(L_Q(p_1,p_2))+v_\pf(\lambda)=v_\pf(L_Q(p_1,p_2)) .
\end{equation}
Pick $m$ big enough so that $\phi^{mn}(R)=Q$. Then  $L_Q  (F_{nm}(r_1,r_2),G_{nm}(r_1,r_2) )=0$.

Let $L_R(x,y)=r_2x-r_1y$ be a linear form defining  $R$, and notice that $L_Q(x,y)$ , $L_R(x,y)$  are factors of $L_Q  ( F_{nm}(x,y),G_{nm}(x,y)   )$. By Gauss's lemma, we can find a polynomial $H(x,y) \in (R_{S})_\pf[x,y]$ such that
$$L_Q  ( F_{nm}(x,y),G_{nm}(x,y)   )=L_R(x,y)L_Q(x,y)H(x,y).$$
Hence
$$v_\pf(L_Q  ( F_{nm}(p_1,p_2),G_{nm}(p_1,p_2)   ))=v_\pf(L_R(p_1,p_2))+v_\pf(L_Q(p_1,p_2))+v_\pf(H(p_1,p_2)).$$
 So by (\ref{Eq2.15})
$$0=v_\pf(L_R(p_1,p_2))+v_\pf(H(p_1,p_2)). $$

Since $v_\pf(L_R(p_1,p_2)) \geq 0$ and $v_\pf(H(p_1,p_2)) \geq 0$ we get $v_\pf(L_R(p_1,p_2))=0$. Finally, since $R$ and $P$ are in normalized form with respect to $\pf$, we have $v_\pf(L_R(p_1,p_2))=\delta_\pf(P,R)=0$  by Remark \ref{Remark2.5}.

\end{proof}

\begin{corollary} \label{cor. S-unit}
Let $\phi$ be an endomorphism of $\PP^1$, defined over $K$. Suppose $\phi$ has good reduction outside $S$. Let $R\in\PP^1(K)$ be a tail point and let $n$ be the period of the periodic part of the orbit of $R$. Let $P\in\PP^1(K)$ be any periodic point that is not $\phi^{mn}(R)$ for some $m$. Then $\delta_\pf(P,R)=0$ for every $\pf \notin S$.
\end{corollary}

\begin{proof}
Take the minimum $m>0$ such that $\phi^{mn}(R)$ is a periodic point.  By Remark \ref{Silverman Theorem 2.18a}, $\phi^n$ also has good reduction outside $S$.

Now apply the previous proposition using $\phi^n$ for $\phi$ , $\phi^{mn}(R)$ for the fixed point and $P$ as the periodic point different from $\phi^{mn}(R)$.
\end{proof}

The last Corollary tells us that $R$ is an $S$-integral point with respect to $P$ (and vice versa). For instance, if $P=[x_1:y_1]$ and $R=[x_2:y_2]$ are written with coprime $S$-integral coordinates, then $x_1y_2-x_2y_1$ is an $S$-unit. Thus, with enough periodic points $P$ (or tail points $R$) we obtain an $S$-unit equation, \emph{i.e.} an equation of the form $au+bv=1 , u,v\in R_S^* , a,b \in K^*$.

The last proposition of this section shows that after slightly enlarging any given set $S$, we can always write a map (or a point) in normalized form with respect to $S$.

\begin{proposition}\label{normalized map}
Let $\phi=[F:G]$  be an endomorphism of $\PP^1$, defined over $K$ with
$$F(X,Y)=a_0X^d+a_1X^{d-1}Y+...+a_{d-1}XY^{d-1}+a_dY^d $$
and
$$G(X,Y)=b_0X^d+b_1X^{d-1}Y+...+b_{d-1}XY^{d-1}+b_dY^d. $$
Then there exists a prime ideal $\pf_0$ of $K$and an element $\alpha\in K$ such that $\phi=[\alpha^{-1}F:\alpha^{-1}G]$ is in normalized form with respect to $S'=S \cup \{\pf_0\}$.
\end{proposition}

\begin{proof}
Consider the fractional ideal $I=(a_0,...,a_d,b_0,...,b_d)R_{S}$. Then by Theorem \ref{ideal class} there is a prime $\pf_I$ of $K$ and $\alpha_I \in K$ such that $I=(\alpha_I) \pf_I R_{S}$.

Consider the representation of $\phi$ given by $\phi=[\alpha_I^{-1}F:\alpha_I^{-1}G]$ and let $S'=S \cup \{\pf_I\}$. Then $v_\pf((\alpha_I^{-1}a_0,...,\alpha_I^{-1}a_d,\alpha_I^{-1}b_0,...,\alpha_I^{-1}b_d))=0$ for every $ \pf \notin S'$. In other words, $[\alpha_I^{-1}F:\alpha_I^{-1}G]$ is normalized with respect to $S'$.
\end{proof}


\begin{proposition} \label{normalized point}
For every $P=[x:y] \in \PP^1(K)$ exists a prime ideal $\pf_0$ of $K$ and an element $\alpha\in K$ such that $P=[\alpha^{-1}x:\alpha^{-1}y]$ is in normalized form with respect to $S'=S \cup \{\pf_0\}$.
\end{proposition}

\begin{proof}
The proof follows the proof of the previous proposition.
\end{proof}


\section{Proof of \Cref{th 3 periodic}}
   \label{main}


For this section, we will state a notation presented in \cite{Canci2007}.

Let  $\mathbf{a_1},...,\mathbf{a_h}$ be a full system of integral representatives for the ideal classes of $R_S$. Hence, for each $i \in \{1,..,h\}$ there is an $S$-integer $\alpha_i \in R_S$ such that
$$ \mathbf{a_i}^h = \alpha_i R_S .$$
Let $L$ be the extension of $K$ given by
$$L=K(\zeta,\sqrt[h] \alpha_1 ,...,\sqrt[h] \alpha_h)$$
where $\zeta$ is a primitive $h$-th root of unity. Consider the following subgroups of $L^*$:
$$\sqrt{K^{*}} : = \{ a \in L^{*} : 	\exists m \in \mathbb{Z}_{> 0} \mbox{  with  } a^m \in K^{*}�\} $$
and
$$\sqrt{R_S^{*}} : = \{ a \in L^{*} : 	\exists m \in \mathbb{Z}_{> 0} \mbox{  with  } a^m \in R_S^{*}�\}. $$

Denote by $\mathbb{S}$ the set of places of $L$ lying above the places in $S$ and by $R_\mathbb{S}$ and $R_\mathbb{S}^{*}$ the ring of $\mathbb{S}$-integers and the group of $\mathbb{S}$-units, respectively in $L$. By definition $R_\mathbb{S}^{*} \cap \sqrt{K^{*}}= \sqrt{R_S^{*}}$ and $\sqrt{R_S^{*}}$ is a subgroup of $L^{*}$ of free rank $s-1$ by Dirichlet's unit theorem.

\begin{lemma} \label{representation}
Assume  the notation above. There exist fixed representations $[x_P:y_P] \in \PP^1(L)$ for every rational point $P \in \PP^1(K)$ satisfying the following two conditions.
\begin{enumerate}

\item [(a)] For every $P\in\PP^1(K)$, we have $x_P,y_P \in \sqrt {K
^{*}}$ and
$$ x_P R_\mathbb{S}+y_P R_\mathbb{S}=R_\mathbb{S}.$$

\item [(b)] If $P,Q\in\PP^1(K)$ then
$$ x_Py_Q-y_Px_Q \in \sqrt{K^{*}} .$$

\end{enumerate}
\end{lemma}

\begin{proof}

Let $P=[x:y]$ be a representation of $P$ in $\PP^1(K)$ and consider  $\mathbf{b} \in \{ \mathbf{a_1},...,\mathbf{a_h} \}$ a representative of $x R_S+y R_S$. We can find   $\beta \in K^{*}$ such that $\mathbf{b}^h =\beta R_S$. Then there is $\lambda  \in K^{*}$ such that
\begin{equation} \label{lemma3.1}
(xR_S +yR_S)^h = \lambda^h\beta R_S.
\end{equation}

We define in $L$
$$ x'=\frac{x}{\lambda \sqrt[h]\beta}  \quad\quad\quad y'=\frac{y}{\lambda \sqrt[h]\beta}$$
and with this definition, it is clear that $x',y' \in \sqrt {K^{*}}$ such that $x' R_\mathbb{S}+y' R_\mathbb{S}=R_\mathbb{S}$.
\\ Furthermore, let $P=[x_1',y_1']$ and $Q=[x_2':y_2']$  where
$$ x_i'=\frac{x_i}{\lambda_i \sqrt[h]\beta_i}  \quad\quad\quad y_i'=\frac{y_i}{\lambda_i \sqrt[h]\beta_i}$$
and $\lambda_i ,\beta_i$ are as the ones described in equation (\ref{lemma3.1}) for $i\in \{1,2\}$. Then
$$(x_1'y_2'-y_1'x_2')^h = \frac{(x_1y_2-y_1x_2)^h}{\lambda_1^h \lambda_2^h \beta_1\beta_2} \in K^{*}. $$
\end{proof}


\begin{proof} [Proof of Theorem \ref{th 3 periodic} part (a)]
Let $P_1,P_2,P_3$ be three different  $K$-rational tail points and let $n_i$ be the period of the periodic part of the orbit of $P_i$ with $i\in\{1,2,3\}$. Let $P$ be a $K$-rational periodic point such that $\phi^{mn_i}(P_i) \neq P$ for every $m\in \mathbb{Z}_{\geq 0}$ and $i\in\{1,2,3\}$ (if such a $P$ does not exist then $|\Per(\phi,K)| \leq 3$ and the proof will be complete).

By Lemma \ref{representation}, for every $i\in\{1,2,3\}$  there exist $P=[x:y]$ , $P_i=[x_i:y_i]$ with $x,y,x_i,y_i\in L$ such that
\begin{enumerate}
\item [(a)] $x_i R_\mathbb{S}+y_i R_\mathbb{S}=R_\mathbb{S} $,

\item [(b)]  $x R_\mathbb{S}+y R_\mathbb{S}=R_\mathbb{S}$,

\item [(c)] $x_iy-y_ix \in \sqrt{K^{*}} $.

\end{enumerate}

By $(a)$ and $(b)$ we have $\delta_{\pf'}(P,P_i)=v_{\pf'}(x_iy-y_ix)$ for every $\pf' \notin \mathbb{S}$ and every $i\in\{1,2,3\}$. Using Corollary \ref{cor. S-unit} we can find $\mathbb{S}$-units $	u_1,u_2,u_3 \in R_\mathbb{S}^{*}$ such that

\begin{equation} \label{eq 1}
x_1y-y_1x =u_1,
\end{equation}
\begin{equation} \label{eq 2}
 x_2y-y_2x =u_2,
\end{equation}
\begin{equation} \label{eq 3}
x_3y-y_3x =u_3.
\end{equation}

Notice that by  $(c)$ , $u_i \in \sqrt{K^{*}}\cap R_\mathbb{S}^{*}=\sqrt{R_S^{*}}$ for each $i\in\{1,2,3\}$.

Using equations (\ref{eq 1}) and (\ref{eq 2}) we get $x=\frac{u_1x_2}{y_2x_1-y_1x_2}-\frac{u_2x_1}{y_2x_1-y_1x_2}$ and $y=\frac{u_1y_2}{y_2x_1-y_1x_2}-\frac{u_2y_1}{y_2x_1-y_1x_2}$. Then by (\ref{eq 3}) we get
$$(x_3y_2-y_3x_2)u_1+(y_3x_1-x_3y_1)u_2 =  (y_2x_1-y_1x_2)u_3. $$
Thus
$$Au+Bv =  1 $$
where $A=\frac{(x_3y_2-y_3x_2)}{(y_2x_1-y_1x_2)}$, $B=\frac{(y_3x_1-x_3y_1)}{(y_2x_1-y_1x_2)}$, $u=u_1u_3^{-1}$ and $v=u_2u_3^{-1}$.

Notice that $A,B \neq 0$ since $P_2 \neq P_3  , P_1 \neq P_3$ and the denominator is not $0$ since $P_1\neq P_2$.

Hence by Corollary \ref{S-unit} with $\Gamma_0=\sqrt{R_S^{*}}$, the total number of solutions $(u,v)\in \sqrt{R_S^{*}}\times \sqrt{R_S^{*}}$ of $Au+Bv=1$ is bounded by $2^{8(2s)}$.

From equations (\ref{eq 1}) and (\ref{eq 3}), we can solve for $x/y$ in terms of $x_1,y_1,x_3,y_3,u$. Therefore there are  $2^{8(2s)}$ possible $[x:y]$. Finally notice that there are at most three periodic points $P$ such that $\phi^{mn_i}(P_i)=P$ for some $m\in\mathbb{Z}_{\geq0}$ and some $i\in\{1,2,3\}$.
Therefore
$$ |\Per(\phi,K)| \leq   2^{16s}+3.$$
\end{proof}

The proof of Theorem \ref{th 3 periodic} part (b) is similar and requires only minor changes at the start and conclusion of the proof.

\begin{proof}[Proof of \Cref{th 4 preperiodic} part (b)]
Let $P_1,P_2,P_3,P_4$ be 4 different $K$-rational periodic  points and let $n_i$ be the period of $P_i$ with $i\in\{1,2,3,4\}$. Let $P$ be a $K$-rational tail point such that $\phi^{m n_i}(P) \neq P_i$ for every $m\in \mathbb{Z}_{\geq 0}$ and $i \in \{1,2,3\}$.

By \Cref{representation} , for every $i\in\{1,2,3\}$  we can take $P=[x:y]$ , $P_i=[x_i:y_i]$ with $x,y,x_i,y_i\in L$ such that
\begin{enumerate}
\item [(a)] $x_i R_\mathbb{S}+y_i R_\mathbb{S}=R_\mathbb{S},$

\item [(b)]  $x R_\mathbb{S}+y R_\mathbb{S}=R_\mathbb{S},$

\item [(c)] $x_iy-y_ix \in \sqrt{K^{*}}.$

\end{enumerate}
Using the same argument of proof of \Cref{th 4 preperiodic} part (a), we get that there are $2^{8(2s)}$ possible $[x:y]$.

Now for the $K$-rational tail points given by $\phi^{m n_1}([x:y])=P_1$,$\phi^{m n_2}([x:y])=P_2$,$\phi^{m n_3}([x:y])=P_3$ we use the same argument with the triples $(P_2,P_3,P_4)$ , $(P_1,P_3,P_4)$ and $(P_1,P_2,P_4)$, respectively. In each case we get the same bound $2^{16s}$.

Therefore,
$$ |\Tail(\phi,K)| \leq  4(2^{16s}).$$
\end{proof}

\begin{proof}[Proof of \Cref{theorem 3}]  

We will prove that we can take a field extension of $K$ to a field $E$ such that $\phi$ has at least three $E$-rational tail points (resp.\ four $E$-rational periodic points) and $[E:K] \leq d^3$. In this case, let $S'$ be the set of places of $E$ lying above the places of $S$. Then the corollary follows by applying \Cref{th 3 periodic} to get 
$$|\Per(\phi,K)| \leq  |\Per(\phi,E)| \leq 2^{16|S'|}+3 = 2^{16|S|d^3}+3 $$
and 
$$|\Tail(\phi,K)| \leq  |\Tail(\phi,E)| \leq 4(2^{16|S'|})=4(2^{16|S|d^3}) . $$
respectively.

Part $(a)$. Assume $\phi$ has at least three periodic points; otherwise the bound trivially holds. By the Riemann-Hurwitz formula a rational function has at most two totally ramified points. Therefore at least one of our periodic points admits a non-periodic preimage. Let $P_1$ be one possible preimage of such a point and consider $E_1$ the field of definition of $P_1$ over $K$. Notice that $[E_1:K]\leq d$. 

Consider $P_2 , P_3 \in \PP^1(\bar{K})$ a preimage of $P_1$ and $P_2$, respectively. Let $E_2$ be the field of definition of $P_2$ over $E_1$ and $E$ the field of definition of $P_3$ over $E_2$. Notice that $[E_2:E_1]\leq d$, $[E:E_2]\leq d$, $P_2 \in \PP^1(E_2)$ and $P_3 \in \PP^1(E)$.


So $[E:K]\leq d^3$ and  $\phi$ has at least three $E$-rational tail points.

Part $(b)$. If  $|\Per(\phi,K)| > 4$ then we can apply \Cref{th 3 periodic} to get the desired bound. Now assume $1 \leq |\Per(\phi,K)| \leq 3$.

Case 1: Suppose there exist a point $P\in \PP^1(K)$ of period 3 under $\phi$. Considering the field extend $E=K(Q)$ of $K$ where $Q$ is a fixed point of $\phi$. Notice that $[E:K] \leq d+1 \leq d^3$ by \Cref{degree dynatomic} and $\phi$ has at least four $E$-rational periodic points.

Case 2: Suppose there exists no periodic point of period 3 in $\PP^1(K)$ but there is a point $P\in \PP^1(\bar{K})-\PP^1(K)$ of period 3 under $\phi$. Considering the field extend $E=K(P)$ of $K$ we have that $\phi$ has a 3-periodic point on $E$. Notice that $[E:K] \leq d^3-d \leq d^3$ by \Cref{degree dynatomic} and $\phi$ has at least four $E$-rational periodic points since $1 \leq |\Per(\phi,K)|$.

Case 3: Suppose there exists no point $P\in \PP^1(\bar{K})$ of period 3 under $\phi$. Then by \Cref{Baker} and \Cref{Baker remark}, $\phi$ admits a point $P_1\in\PP^1(\bar{K})$ of period 2 and two distinct fixed points $P_2,P_3\in\PP^1(\bar{K})$. Since $1 \leq |\Per(\phi,K)| \leq 3$ we can assume that at least one of $P_1,P_2,P_3$ is $K$-rational. Let $E=K(P_1,P_2,P_3)$. Notice that $[E:K]\leq d^3$ by \Cref{degree dynatomic} and $|\Per(\phi,E)| \geq 4$.

\end{proof}

After assuming $|\Tail(\phi,K)| \geq 3$  ($|\Per(\phi,K)| \geq 4$), \Cref{th 4 preperiodic} $(a)$ and $(b)$ provides a bound for $|\Per(\phi,K)|$ ( $|\Tail(\phi,K)|$) independent of the degree of $\phi$. We claim that in order to get bounds for $|\Per(\phi,K)|$ and $|\Tail(\phi,K)|$ independent of the degree of $\phi$, the hypotheses  $|\Tail(\phi,K)| \geq 3$ and $|\Per(\phi,K)| \geq 4$ are required. This can be seen in section~\ref{Eg} where we provide a couple of examples that show our claim.

In order to improve the bounds given in this section, we have to overcome two technical obstacles:
\begin{enumerate}
\item [(a)] Due to the possibility of a nontrivial class group, every $P \in \PP^1(K)$ cannot always be written as $P=[x:y]$ with $x$ and $y$  coprime $S$-integers.

\item [(b)] In order to apply \Cref{th 4 preperiodic}, we extended the field $K$ to have enough $K$-tail (or $K$-periodic) points. However, after doing so, the degree of the rational function appears in the exponent of our bound.
\end{enumerate}

We overcome $(a)$ by analyzing the ideal class generated by $x$ and $y$ in $R_S$, when $P=[x:y]$ is preperiodic. To overcome $(b)$ we use the theory of Thue-Mahler equations, instead of $S$-unit equations, to avoid having to extend the field $K$. After using Thue-Mahler equations we will obtain that the degree of the rational function appears in a polynomial way in our bound. We will provide solutions to problems $(a)$ and $(b)$ in the next section.


 \section{Proof of \Cref{theorem 4} and \Cref{theorem 5} using Thue-Mahler equations}
   \label{main2}

First we will prove \Cref{theorem 4}. Assume the hypotheses in \Cref{theorem 4}.

Notice that if $\phi$ has at least four $K$-rational periodic points, then by \Cref{th 4 preperiodic}
$$|\Tail(\phi,K)| \leq 4(2^{16s}).$$
Therefore until the end of the proof of \Cref{theorem 4} we assume $|\Per(\phi,K)|  \leq 3$.

If $|\Per(\phi,K)|=0$ then  $|\Tail(\phi,K)|=0$. So there is nothing to prove in this case. The remaining possibilities can be divided into two cases: when $|\Per(\phi,K)|  =2$ or  $3$  and when $|\Per(\phi,K)|  = 1$.

Before we start analyzing these two cases, we will prove a proposition that will be useful in both.
\begin{proposition} \label{prop to use T-M}
Let $K$ be a number field and $S$ a finite set of places of $K$ containing all the archimedean ones. Let $\phi$ be an endomorphism of $\PP^1$, defined over K and $d \geq 2$ the degree of $\phi$. Assume $\phi$ has good reduction outside $S$ and $\phi$ admits a normalized form with respect to $S$. Let $\mathscr{A}\subset \Tail(\phi,K)$ be such that every point in $\mathscr{A}$ admits a normalized form with respect to $S$. Then
$$|\mathscr{A}| \leq  \max\left \{  (5 * 10^6 (d^3+1))^{s+1} ,4(2^{64s})\right \}.$$
\end{proposition}
 \begin{proof}
Suppose that there exists a $\bar{K}$-rational periodic point  $P_{*}$ of period $1,2$ or $3$ such that $[E:K] \geq 3$ where $E=K(P_{*})$. Notice that $[E:K] \leq d^3$.

Let $S_E$ be the set of places of $E$ lying above places in $S$. Applying  \Cref{normalized point} to $P_{*}$ and $S_E$, we can find a prime $\pf_E$ in $E$ such that $P_{*}$ can be written in normalized form with respect to $S_E \cup \{\pf_E\}$. Consider  $S'=S \cup \{\pf_K\}$ where $\pf_K$ is the prime of $K$ lying below $\pf_E$ and let $S_E' $ be the set of places in $E$ lying above places in $S'$.

Let $P=[x:y] \in \mathscr{A}$ be in normalized form with respect to $S'$ and  $P_{*}=[a:b]\in \PP^1(E)$ in normalized form with respect to $S_E'$. Notice that $P_{*}$ is not in the orbit of $P$ since it is not $K$-rational. 

For every prime $\pf_E' \notin S_E'$, $\delta_{\pf_E'}(P,P_{*})=0$. Then for every $\pf_E' \notin S_E' $
\begin{equation}\label{eq in T-H}
v_{\pf_E'}(ay-bx)=0.
\end{equation}

Denote by $N_{E/K}$ the norm from $E$ to $K$ and consider $F(X,Y)=N_{E/K}(aY-bX)\in K[X,Y]$ where the embedding of $E$ over $K$ act trivially on $X$ and $Y$. Since $P_{*}$ is in normalized form with respect to $S_E'$, we have that $a,b\in R_{E,S_E'} $. Hence $F(X,Y)\in R_{K,S'}[X,Y]$.  Notice that the degree of $F$ is $[E:K]$. Since $P_{*}$ is a root of $F(X,Y)$ and $E$ is the field of definition of $P_{*}$ we have that $F(X,Y)$ is irreducible over $K$. Finally using that every $P=[x:y]\in \mathscr{A}$ is in normalized form with respect to $S^{'}$ and equation (\ref{eq in T-H}) we have $F(x,y) \in R_{K,S'}^{*}$.

Now we have all the hypotheses to apply \Cref{Thue-Mahler}. Therefore in this case we get
$$|\mathscr{A}| \leq (5 * 10^6 [E:K])^{s+1} \leq (5 * 10^6 d^3)^{s+1}.$$
Now suppose that for every $\bar{K}$-periodic point $P$ of period $1,2$ or $3$, we have $[K(P):K] \leq 2$. We claim that in this case we can find a field $E$ of degree $[E:K] \leq 4$ such that $\phi$ has at least 4 distinct $E$-rational periodic points.
To prove the claim we just need to use \Cref{Baker} and \Cref{Baker remark} as follows.

\begin{enumerate}
\item[Case 1:] There exists a point $P\in \PP^1(\bar{K})$ of period 3 under $\phi$. Let $Q\in \PP^1(\bar{K})$ be a fixed point of $\phi$ and $E=K(P,Q)$. Then by assumption $[E:K] \leq 4$ and we have $|\Per(\phi,E)| \geq 4$.

\item[Case 2:] There does not exist a point $P\in\PP^1(\bar{K})$ of period 3 under $\phi$. By \Cref{Baker} and \Cref{Baker remark}, $\phi$ admits a point $P_1\in\PP^1(\bar{K})$ of period 2 and two distinct fixed points $P_2,P_3�\in\PP^1(\bar{K})$. Since $1 \leq |\Per(\phi,K)|\leq 3$, we can assume that at least one of $P_1,P_2,P_3$ is $K$-rational. Let $E=K(P_1,P_2,P_3)$. Then again we have $[E:K] \leq 4$ and $|\Per(\phi,E)| \geq 4$.
\end{enumerate}

Then by \Cref{th 4 preperiodic}
$$|\mathscr{A}|  \leq |\Tail(\phi,K)| \leq |\Tail(\phi,E)| \leq 4(2^{16(4(s))})=4(2^{64s}).$$

In any case
$$|\mathscr{A}|  \leq  \max\left\{  (5 * 10^6 (d^3+1))^{s+1} ,4(2^{64s}) \right \}.$$

 \end{proof}

Notice that if $R_S$ is a PID then \Cref{theorem 4} follows immediately from  \Cref{prop to use T-M}.

  \begin{proof}[Proof of \Cref{theorem 4}]


\textbf{Case 1:} $|\Per(\phi,K)|  \in \{2,3\}$

By \Cref{normalized map} we can assume $\phi$ is in normalized form with respect to $S_1$, for some $S_1$ with $|S_1|=|S|+1$ and $S \subset S_1$.

Let $P_1=[x_1:y_1],P_2=[x_2:y_2]$ be two different $K$-rational periodic points. For every $P=[x_P:y_P] \in  \Tail(\phi,K)$ there is $i_P\in \{1,2\}$ such that
$$\delta_\pf(P,P_{i_P}) =0 \quad  \mbox{for every} \quad \pf \notin S_1  .$$
Then
$$  (xy_{i_P}-yx_{i_P})R_{K,S_1}=(x_P,y_P)(x_{i_P},y_{i_P})R_{K,S_1} \quad \mbox{for every} \quad P \in \Tail(\phi,K). $$

Applying \Cref{normalized point} on $P_1,P_2$ and $S_1$, we can find a representation of $P_1$ and $P_2$ such that $P_1=[x_1':y_1']$ and $P_2=[x_2':y_2']$ are in normalized form with respect to $S_2$, for some $S_2$ with $S_1 \subset S_2$ and $|S_2| =|S_1|+2$. Hence, for every $P \in \Tail(\phi,K) $
$$(xy'_{i_P}-yx'_{i_P})R_{K,S_2}=(x_P,y_P)R_{K,S_2}$$
and $x_P$ and $y_P$ generate a principal $R_{K,S_2}$-ideal. Therefore, for every $P \in  \Tail(\phi,K)$ we can find a representation of $P$ that is normalized with respect to $S_2$, namely $P=[\alpha^{-1}_Px_P:\alpha^{-1}_Py_P]$, where $\alpha_P=x_Py'_{i_P}-y_Px'_{i_P}$ (\Cref{normalize=principal}).

Every point $P \in  \Tail(\phi,K)$ admits a normalized form with respect to $S_2$  and  $\phi$ is in normalized form with respect to $S_2$ with good reduction outside $S_2$. Applying \Cref{prop to use T-M} gives
$$|\Tail(\phi,K)| \leq   \max\left\{  (5 * 10^6 (d^3+1))^{s+4} ,4(2^{64(s+3)})\right \}.$$

\textbf{Case 2:} $|\Per(\phi,K)|  =1$

By \Cref{normalized map} we can assume $\phi$ is in normalized form with respect to $S_1$, for some $S_1$ with $|S_1|=|S|+1$ and $S \subset S_1$. Let $Q\in \PP^1(K)$ be the only $K$-rational periodic point.  Applying \Cref{normalized point} on $Q$ and $S_1$, we can find a representation of $Q$ such that $Q=[q_1:q_2]$ is in normalized form with respect to $S_2$, for some $S_2$ with $S_1 \subset S_2$ and $|S_2| =|S_1|+1$.

Let $P=[x_P:y_P] \in \Tail(\phi,K)$. Since $\phi=[F,G]$ is in normalized form with respect to $S_2$ and $\phi$ has good reduction outside $S_2$,
Thus 
$$v_\pf( (F(x_P,y_P),G(x_P,y_P)) )=v_\pf((x_P,y_P)^d)\quad \mbox{for every}\quad\pf \notin S_2.$$
Therefore, 
\begin{equation} \label{eq emma}
(F(x_P,y_P),G(x_P,y_P))=(x_P,y_P)^d  \quad\mbox{for every as}\quad R_{K,S_2}\mbox{-ideals.} 
\end{equation}
Applying the last equality repeatedly we get that the $R_{K,S_2}$-ideal class $[(x_P,y_P)^{d^{n}}]=[(q_1,q_2)]=[1]$  is trivial for some $n >0$ depending on $P$.

Assume the notation of \Cref{dynatomic}. By \Cref{dynatomic} there are at most two values of $n$ such that
$$a_{Q}^{*}(n)=\ord_{Q}(\Phi^{*}_{\phi,n}(X,Y)) \neq 0. $$
Since $a_{Q}^{*}(1) \neq 0$ we get that either $a_{Q}^{*}(2) = 0$ or $a_{Q}^{*}(3) = 0$. Set $l=\min \{i :a_{Q}^{*}(i) = 0\}$.

Consider $\Phi^{*}_{\phi,l}(X,Y)$ and notice that every root  of $\Phi^{*}_{\phi,l}$ is a periodic point of period $1$ or $l$, different from $Q$. Let

$$\Phi^{*}_{\phi,l}(X,Y)=cf_1(X,Y)^{\alpha_1}\cdots f_i(X,Y)^{\alpha_i}\cdots f_r(X,Y)^{\alpha_r} $$
be the irreducible factorization of $\Phi^{*}_{\phi,l}(X,Y)$ over $K$ and $c\in K^{*}$. Let $e_i=\deg f_i$ for $i=1,...,r$. Note that the degree of $\Phi^{*}_{\phi,l}$ is $d^l-d$.

Fix $i\in\{1,...,r\}$. Let $Q_i=[a_i:b_i] \in\PP^1( \bar{K})$ be a root of $f_i(X,Y)$. Consider $E_i=K(Q_i)$ the field of definition of $Q_i$ and $e_i=[E_i:K]$. Let $S_{E_i}$ be the set of places of $E_i$  lying above places of $S_2$.

Denote by $N_{E_i/K}$ the norm from $E_i$ to $K$ and notice that $f_i(X,Y)=N_{E_i/K}(a_iY-b_iX)\in K[X,Y]$ up to a constant. For every $P \in \Tail(\phi,K)$ and for every $\pf_{E_i} \notin S_{E_i}$ we have $ \delta_{\pf_{E_i}}(P,Q_i)=0. $
Then
\begin{equation}\label{4.aux}
(x_Pb_i-y_Pa_i)=(a_i,b_i)(x_P,y_P) \mbox{   as $R_{E_i,S_{E_i}}$-ideals.}
\end{equation}


Applying $N_{E_i/K}$ to (\ref{4.aux}) we get

\begin{equation}\label{5.aux}
(f_i(x_P,y_P))R_{K,S_2}=I_i (x_P,y_P)^{e_i}R_{K,S_2}
\end{equation}
where $I_i=N_{E_i/K}((a_i,b_i))$ is an $R_{K,S_2}$-ideal. Taking appropriate powers and multiplying over all $i$ gives
\begin{equation}\label{6.aux}
(\Phi^{*}_{\phi,l}(x_P,y_P))R_{K,S_2}=I (x_P,y_P)^{\sum_i \alpha_i e_i}R_{K,S_2}
\end{equation}
where $I=\Pi_i I_i^{\alpha_i}$ is an $R_{K,S_2}$-ideal.

By \Cref{ideal class} applied to the $R_{K,S_2}$-ideal $I$, there is a prime ideal $\pf_0$ in $K$ and $\beta_I \in K $ such that $(\beta_I)I=\pf_0R_{K,S_2}$.
Consider $S_2'=S_2\cup \{ \pf_0 \}$ then multiplying  (\ref{5.aux}) by $\beta_I$ we get

$$\beta_I (\Phi^{*}_{\phi,l}(x_P,y_P))R_{K,S_2}=\beta_I I(x_P,y_P)^{d^l-d}R_{K,S_2}=\pf_0R_{K,S_2} (x_P,y_P)^{d^l-d}R_{K,S_2}.$$
Notice that $\pf_0R_{K,S_2}$ is the trivial ideal in $R_{K,S_2'}$. Therefore
\begin{equation}
\beta_I (\Phi^{*}_{\phi,l}(x_P,y_P))R_{K,S_2'}= (x_P,y_P)^{d^l-d}R_{K,S_2'}.
\end{equation}

Thus, the ideal class of $(x_P,y_P)^{d^l-d}$ in $R_{K,S_2'}$ is trivial. Then the ideal class of $(x_P,y_P)^{d^{n}}$ in $R_{K,S_2'}$ is trivial since the ideal class of $(x_P,y_P)^{d^{n}}$ in $R_{K,S_2}$ is trivial . Taking the g.c.d. of $d^l-d$  and $d^{n}$ we get that the ideal class of $(x_P,y_P)^d$ in $R_{K,S_2'}$ is trivial.

Let $\mathscr{A}$ be the set of all $K$-rational tail points excluding the initial point in each maximal orbit. Using equation (\ref{eq emma}) and \Cref{normalize=principal} every point $P \in \mathscr{A}$ admits a normalized form with respect to $S_2'$.

Now applying \Cref{prop to use T-M} to $\mathscr{A}$ and $S_2'$, we get
$$|\mathscr{A}|  \leq  \max\left\{  (5 * 10^6 (d^3+1))^{s+4} ,4(2^{64(s+3)}) \right\}.$$
This gives us
$$|\Tail(\phi,K)|\leq d|\mathscr{A}|   \leq  d\max\left \{  (5 * 10^6 (d^3+1))^{s+4} ,4(2^{64(s+3)}) \right \}.$$

 \end{proof}


Now we will prove \Cref{theorem 5}.

Assume the hypotheses in \Cref{theorem 5}. Hence  $|\Tail(\phi,K)| \geq 1 $.

Notice that if $\phi$ has at least three $K$-rational tail points, then by \Cref{th 3 periodic} we have that
$$|\Per(\phi,K)| < 2^{16s}+3. $$

Therefore in the rest of the section we assume $|\Tail(\phi,K)| \in \{1,2\} $.

As before, we will need to prove a proposition to use in the proof of \Cref{theorem 5}.
\begin{proposition} \label{prop to use T-M2}
Let $\phi$ be an endomorphism of $\PP^1$, defined over K. Let $d \geq 2$ be the degree of $\phi$. Assume $\phi$ has good reduction outside $S$ and $\phi$ is in normalized form with respect to $S$. Let $\mathscr{A}�\subset \Per(\phi,K)$ such that every point in $\mathscr{A}$ admits a normalized form with respect to $S$. Then
$$|\mathscr{A}| \leq  \max\left\{  (5 * 10^6 (d-1))^{s+1} ,4(2^{128s})\right \}.$$

\end{proposition}

 \begin{proof}
Suppose that for every tail point $P_{*} \in \PP^1(\bar{K}) -\PP^1(K)$ such that $\phi(P_{*})$ is a $K$-rational periodic point, $[K(P_{*}):K] \geq 3$ where $E=K(P_{*})$ is the field of definition of $P_{*}$. Then the same proof as the first part of the proof of \Cref{prop to use T-M} yields the desired result ( notice that $[E:K] \leq d-1$).

Now suppose that for every tail point $P_{*} \in \PP^1(\bar{K}) -\PP^1(K)$ such that $\phi(P_{*})$ is a $K$-rational periodic point, $[K(P_{*}):K] < 3$. In this case, assume we can find three different $K$-rational periodic points $Q_1,Q_2,Q_3$ (otherwise $ \Per(\phi,K)$ is trivially bounded ). We can find three different tail points $P_i \in \PP^1(\bar{K}) - \PP^1(K)$ such that $\phi(P_i)=Q_i$ and $1 \leq [K(P_i):K] \leq 2$ where $1\leq  i \leq 3$. Applying \Cref{th 4 preperiodic} gives
 $$|\mathscr{A}| \leq  4(2^{128s}).$$

Therefore we get
$$|\mathscr{A}| \leq  \max\left \{  (5 * 10^6 (d-1))^{s+1} ,4(2^{128s})\right \}.$$

 \end{proof}

Notice that if $R_S$ is a PID then every point in $\Per(\phi,K)$ admits a normalized form with respect to $S$. Thus, in this case \Cref{prop to use T-M2} gives a bound for $\Per(\phi,K)$ and the hypothesis on the existence of a $K$-rational tail point will not be required in \Cref{theorem 5}.

  \begin{proof}[Proof of \Cref{theorem 5}]
  Similarly, as in the proof of Case I of \Cref{theorem 4} we the only change of using \Cref{prop to use T-M2} instead of \Cref{prop to use T-M} to obtain
$$|\Per(\phi,K)| \leq   \max\left \{  (5 * 10^6 (d-1))^{s+3} ,4(2^{128(s+2)})\right \}+1.$$
 \end{proof}


\section{Examples}
   \label{Eg}

In this section we will present two examples that show the sharpness of the hypotheses of \Cref{th 3 periodic} part $(a)$ and $(b)$.

The first example gives a family of rational functions with exactly two $\mathbb{Q}$-rational tail points and a fixed set of places of bad reduction. However the size of the set of $\mathbb{Q}$-rational periodic points grows with the degree of the rational functions in the family. This proves that the  hypothesis of \Cref{th 3 periodic} part $(a)$ is necessary.

\begin{example} \label{example2}
Consider 
\begin{equation*}
f_d(x)=\displaystyle\frac{1}{x}+\displaystyle\frac{(x-2^{-d})(x-2^{-d+1})...(x-1)...(x-2^{d-1})(x-2^d)}{x^{2d+1}} \quad \in \mathbb{Q}(x). 
\end{equation*}

If we take $S= \{ \infty , 2 \}$ then $f_d(x)$ has good reduction outside $S$.

Now notice that $0$ and $\infty$ are tail points and $1$ is a fixed point with orbit $0 \rightarrow \infty \rightarrow 1 \rightarrow 1$. Also for every $i\in \left \{-d,..,-1,1,..,d \right \} $ the points $2^{i}$ are $\mathbb{Q}$-rational periodic points of period 2.

Finally by  \Cref{th 3 periodic} if $d$  is large enough, then $\Tail(f_d,\mathbb{Q})=\{0,\infty\}$. Thus, this gives an example of a family of rational functions $f_d$ such that each rational function $f_d$ has exactly two $\mathbb{Q}$-rational tail points, good reduction outside of a fixed finite set of places $S$, and the number of $\mathbb{Q}$-rational periodic points grows with the degree of $f_d(x)$.
\end{example}

The second example gives a family of rational functions with exactly three $\mathbb{Q}$-rational periodic points and a fixed set of places of bad reduction. However the size of the set of $\mathbb{Q}$-rational tail points grows with the degree of the rational functions in the family. This proves that the  hypothesis of \Cref{th 3 periodic} part $(b)$ is necessary.

\begin{example} \label{example1}
Consider 
\begin{equation*}
f_d(x)=\displaystyle\frac{(x-1)(x-2)(x-2^2)...(x-2^{d-1})}{x^d} \quad \in \mathbb{Q}(x).
\end{equation*}

If we take $S= \{ \infty , 2 \}$ then $f_d(x)$ has good reduction outside $S$.

Now we notice that 0 is a periodic point with orbit $0 \rightarrow \infty \rightarrow 1 \rightarrow 0$ and that $2,..,2^{d-1}$ are in the tail of 0.

Finally by  \Cref{th 4 preperiodic} if $d$ is large enough, then $\Per(f_d,\mathbb{Q})=\{0,1,\infty\}$. Thus, this gives an example of a family of rational functions $f_d$ such that each rational function $f_d$ has exactly three $\mathbb{Q}$-rational periodic points, good reduction outside of a fixed finite set of places $S$, and the number of $\mathbb{Q}$-rational tail points grows with the degree of $f_d(x)$.
\end{example}

\bibliographystyle{plain}

\bibliography{bibfile}

\end{document}